\documentclass{article}%
\usepackage{amsfonts}
\usepackage{amsmath}
\usepackage{amsthm}
\usepackage{amssymb}
\usepackage{bm}
\usepackage{enumitem}
\usepackage{fullpage}
\usepackage{color}
\usepackage{pstricks}
\usepackage{pstricks-add}
\usepackage{graphicx,overpic,psfrag}
\usepackage{graphicx}
\usepackage{tikz}%
\providecommand{\U}[1]{\protect\rule{.1in}{.1in}}
\allowdisplaybreaks
\newtheorem{theorem}{Theorem}
\newtheorem{lemma}[theorem]{Lemma}

\newtheorem{proposition}[theorem]{Proposition}
\newtheorem{definition}[theorem]{Definition}
\newtheorem{notation}[theorem]{Notation}
\newtheorem{example}[theorem]{Example}
\newtheorem{remark}[theorem]{Remark}

\begin{document}

\title{Number of common roots and resultant \\of two tropical univariate polynomials}
\author{Hoon Hong\\Department of Mathematics \\North Carolina State University, USA
\and J. Rafael Sendra\\Grupo \textsc{asynacs}, \\Dpto. de F\'{\i}sica y Matem\'aticas, \\Universidad de Alcal\'a, 28871-Alcal\'a de Henares, Madrid, Spain }
\date{}
\maketitle

\begin{abstract}
It is well known that for two univariate polynomials over complex number field
the number of their common roots is equal to the order of their resultant. In
this paper, we show that this fundamental relationship still holds for the
tropical polynomials under suitable adaptation of the notion of order, if the
roots are simple and non-zero.

\end{abstract}

 \vspace*{2mm}


\section{Introduction}


The resultant plays a crucial role in algebra and algebraic geometry
~\cite{syl53, Sal85, Lo83,CLO97,BL2010,DKS2013}. Let $m,n$ be fixed. Let
\begin{align*}
\mathbf{A}  &  =\mathbf{a}_{0}\mathbf{x}^{m}+\mathbf{a}_{1}\mathbf{x}%
^{m-1}+\cdots+\mathbf{a}_{m}\in\mathbb{C}\left[  \mathbf{a,x}\right] \\
\mathbf{B}  &  =\mathbf{b}_{0}\mathbf{x}^{n}+\mathbf{b}_{1}\mathbf{x}%
^{n-1}+\cdots+\mathbf{b}_{n}\ \in\mathbb{C}\left[  \mathbf{b,x}\right]
\end{align*}
Then the \emph{resultant} $R\in\mathbb{C}[\mathbf{a},\mathbf{b}]$ is defined
as the smallest monic polynomial (w.r.t. a given order) such that, for every
$a\in\mathbb{C}^{m+1}$ and $b\in\mathbb{C}^{n+1}$, if the two polynomials
$\mathbf{A}\left(  a,\mathbf{x}\right)  ,\mathbf{B}\left(  b,\mathbf{x}%
\right)  \in\mathbb{C}[\mathbf{x}]$ have a common complex root then
$R(a,b)=0$.\footnote{
It is well-known that the resultant $R$ can be defined  in various other ways: for instance, in terms of Sylvester matrix, Bezout matrix, Barnett matrix, Hankel matrix (see \cite{DG04} for a nice summary). Those definitions are  more useful for computational purposes. However they are  also more  complicated when deducing theoretical or structural properties.  Since the main interest of this paper is not computational but structural, we chose the more structural definition.}
We recall the following two well known fundamental properties of resultants. Let
$a\in\mathbb{C}^{m+1}\ $and $b\in\mathbb{C}^{n+1}$ such that $a_{0},b_{0}%
\neq0.$ Then we have

\begin{enumerate}
\item[P1.] The point $\left(  a,b\right)  $ is a root of $R$ if and only if the
polynomials $\mathbf{A}\left(  a,\mathbf{x}\right)  $ and $\mathbf{B}\left(
b,\mathbf{x}\right)  $ have a common complex root.
(Of course, the `if' part is immediate from the definition and thus the interesting part is the `only if').

\item[P2.] The order of the point $(a,b)$ at $R$ is equal to the number of
common complex roots of the polynomials $\mathbf{A}\left(  a,\mathbf{x}%
\right)  $ and $\mathbf{B}\left(  b,\mathbf{x}\right) .$ (See the appendix)
\end{enumerate}

A natural question arises: \emph{whether/how these properties can be adapted
to polynomials over the tropical semifield.} Recall that the tropical
semifield\/ is the set $\mathbb{R\cup}\left\{  -\infty\right\}  $ where the
addition operation is defined as the usual maximum, the multiplication
operation is defined as the usual addition. As a result, it does not allow
subtraction (due to lack of additive inverse; hence the name semifield). It
has been intensively investigated due to numerous interesting applications
~\cite{Si88, Pi98, PaSt04, HeOlWo06, BoJenSpSt07, GoMi08, SpSt09, ItMiSh09,
But10, StTr13,LitSer14, BrSh14, MacStu15,GS2016}.

Note that, unlike polynomials over $\mathbb{C}$, it is easy to compute the
roots of polynomials over tropical semifield, and thus, counting number of
common roots is also easy. Hence the motivation for asking the above question
is not for finding an efficient algorithm, but for gaining structural
understanding on the relation between roots and the resultant.


There have been several adaptations of resultant over $\mathbb{C}$ to the
tropical semifield ~\cite{MiGr06, DFS07, BoJenSpSt07, Ta08, Od08, JY13}. In
particular, Tabera~\cite{Ta08} and Odagiri~\cite{Od08} showed that the above
property P1\ holds over the tropical semifield, if one redefines the notions
of roots and resultant as follows: (1) a root is redefined as a point where
the graph of the polynomial is not smooth (2) the resultant is redefined as
the tropicalization of resultant over $\mathbb{C}$ (\cite{Ta08}) or as the
permanent of the Sylvester matrix (\cite{Od08}). In this paper, we will follow
the definition of \cite{Ta08}.

The main contribution of this paper is to show that the property P2 also holds
over the tropical semifield, if one puts a slight restriction on $\left(
a,b\right)  $ and if one makes a suitable adaptation of the notion of order,
as follows: (1) we restrict $\left(  a,b\right)  $ such that the polynomials
$\mathbf{A}\left(  a,\mathbf{x}\right)  $ and $\mathbf{B}\left(
b,\mathbf{x}\right)  $ have only simple and non-zero roots. (2) the notion of
the order of a point $p$ at a multivariate polynomial $C$ is replaced by the
new concept of ``order'', which is the $\log_{2}$ of the numbers of terms, say
$t,$ in $C$ such that $t\left(  p\right)  =C(p)$.

The paper is structured as follows. In Section \ref{sec:result}, we state
formally the main result of the paper. In Section \ref{sec:proof}, we provide
a proof. In Section \ref{sec:conclude}, we summarize the main result and
discuss some potential generalizations and associated difficulties. In
Appendix, we include a simple proof of P2 over $\mathbb{C}$, provided by
Laurent Bus\'e.

\section{Main Result}

\label{sec:result}

In this section, we present the main result of this paper. For this, we need
to recall some basic notions on the tropical semi-field and the tropical resultant.
The \emph{tropical semi-field} is the tuple $(\mathbb{T},+,\times,/)$, where $\mathbb{T=R\cup}\left\{  -\infty\right\}  $,
$+$ is the usual maximum, $\times$ is the usual addition, and $/$ is the usual
subtraction.  It is easy to see that the  additive identity (tropical zero) is~$-\infty$ and that the
multiplicative identity (tropical one) is~$0$.

Let $\mathbb{T}\left[  \mathbf{x}\right]  $ be the set of all polynomials in
the indeterminate~$\mathbf{x}$. The polynomial $C\in\mathbb{T}\left[
\mathbf{x}\right]  $ represents a function: $\mathbb{T\rightarrow T}$. We say
that $\alpha\in\mathbb{T}$ is a \emph{root} of $C$ if the graph of $C$ has a
corner over $\alpha.$ The \emph{multiplicity} of $\alpha$ is the change in the
slopes of the graph across $\alpha.$ When the multiplicity of $\alpha$ is one,
we say that $\alpha$ is \emph{simple}.

Finally, we recall the notion of tropical resultant (see \cite{Ta08} for
further details). Let $R\in\mathbb{C}[\mathbf{a},\mathbf{b}]$ be the resultant
w.r.t. the fixed degrees $m,n$. Then the \emph{tropical resultant}%
\textsf{\ }$\mathfrak{R}\in\mathbb{T}[\mathbf{a},\mathbf{b}]$ is defined as
the tropicalization of $R$, that is, the tropical sum of the supports of
$R$.\footnote{A motivation behind this definition comes from the observation
that if one carries out the $\log_{t}$-coordinate transform where $t
\rightarrow\infty$, then $R$ becomes the tropical sum of the support of $R$.}

Now we adapt the notion of the order to the tropical semifield.

\begin{definition}
[Order]Let $C\in\mathbb{T}\left[  \mathbf{z}_{1},\ldots,\mathbf{z}_{l}\right]
$ be a tropical non-zero polynomial. Let $C={t}_{1}+\cdots+ {t}_{r}$
where~${t}_{i}$'s are terms in $\mathbf{z}$ with tropical non-zero
coefficients. Let $p\in\mathbb{T}^{l}.$ Let $E_{C}(p):=\left\{  {t}%
_{i}:C\left(  p\right)  ={t}_{i}\left(  p\right)  \right\}  $. The
\emph{order}\footnote{Over $\mathbb{T}$, the notions of multiplicity and order
are \emph{not} the same, unlike over $\mathbb{C}$, already in the univariate
case.} of~$p$ in $C$, written as $O_{C}\left(  p\right)  $, is defined by
\[
O_{C}\left(  p\right)  := \log_{2} \#E_{C}(p)
\]

\end{definition}

\begin{example}
\label{ex:index-order} Let $C(z)=\mathbf{z}_{1}\mathbf{z}_{2}+2\mathbf{z}%
_{1}+2\in\mathbb{T}[\mathbf{z}]$. The plots in Figure \ref{fig:index-order}
describe the graph of $C$. \begin{figure}[h]
\begin{center}
\includegraphics[width=12cm]{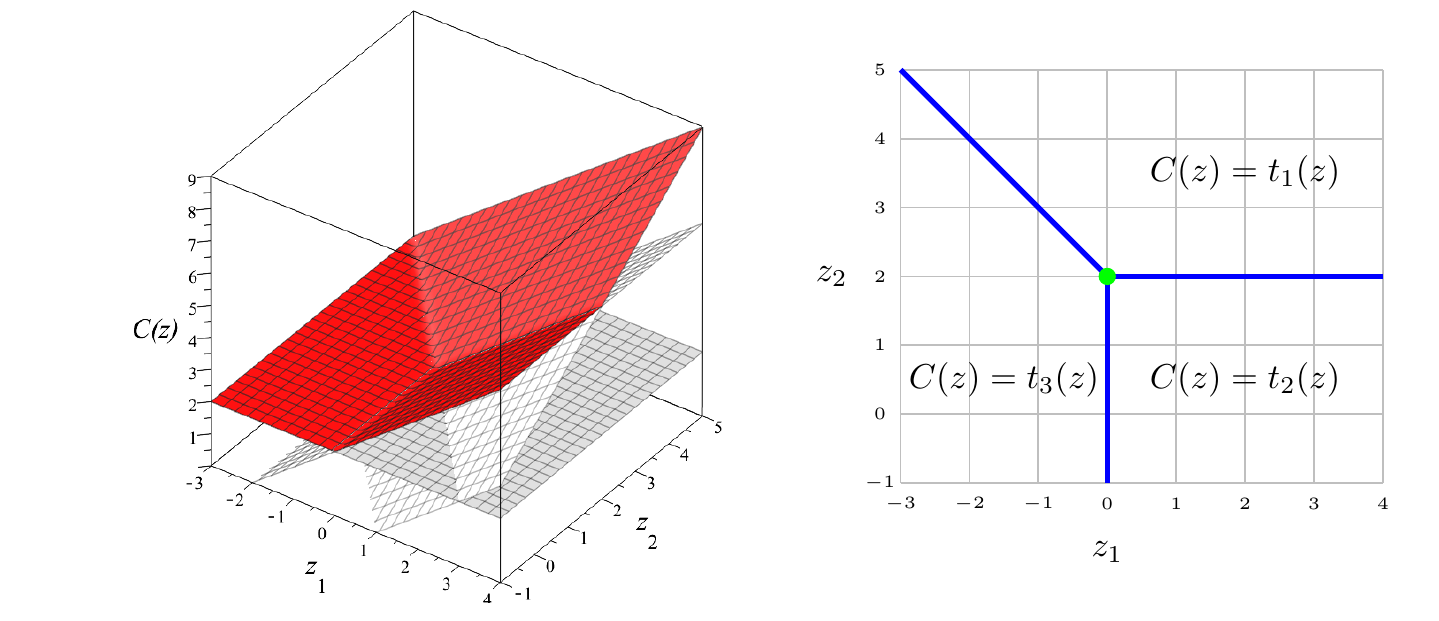}
\par%
\begin{tabular}
[c]{|l|l|l|c|c|}\hline
$(p_{1},p_{2})\in\mathbb{T}^{2}$ & color & $E_{C}(p)$ & $\# E_{C}(p)$ &
$O_{C}(p)$\\\hline\hline
$t_{1}=t_{2}=t_{3}$ & green & $\{t_{1},t_{2},t_{3}\}$ & $3$ & $\log_{2}%
3$\\\hline
$t_{1}=t_{2}>t_{3}$ & blue & $\{t_{1},t_{2}\}$ & $2$ & $1$\\\hline
$t_{1}=t_{3}>t_{2}$ & blue & $\{t_{1},t_{3}\}$ & $2$ & $1$\\\hline
$t_{2}=t_{3}>t_{1}$ & blue & $\{t_{2},t_{3}\}$ & $2$ & $1$\\\hline
$t_{1}>t_{2},t_{3}$ & white & $\{t_{1}\}$ & $1$ & $0$\\\hline
$t_{2}>t_{1},t_{3}$ & white & $\{t_{2}\}$ & $1$ & $0$\\\hline
$t_{3}>t_{1},t_{2}$ & white & $\{t_{3}\}$ & $1$ & $0$\\\hline
\end{tabular}
\ \ \
\end{center}
\caption{Polynomial $C$ in Example \ref{ex:index-order}}%
\label{fig:index-order}%
\end{figure}

\noindent The left plot shows the function represented by~$C$. The red
piecewise plane is the graph of~$C(\mathbf{z})$ and the three gray planes are
the graphs of the terms in~$C.$ Let $t_{1}=\mathbf{z}_{1}\mathbf{z}_{2}$,
$t_{2}=2\mathbf{z}_{1}$, and $t_{3}=2$ be the terms in $C.$ The right plot
explicitly shows the partition of the~$\mathbf{z}$ plane into pieces and the
corresponding term for each piece. The bottom table shows the values of the
order on each piece.
\end{example}

\begin{theorem}
[Main Result]\label{thm:main} Let
\begin{align*}
A  &  =a_{0}\mathbf{x}^{m}+a_{1}\mathbf{x}^{m-1}+\cdots+a_{m}\in
\mathbb{T}\left[  \mathbf{x}\right]  ,\ \ a_{0}\neq-\infty\\
B  &  =b_{0}\mathbf{x}^{n}+b_{1}\mathbf{x}^{n-1}+\cdots+b_{n}\;\;\;\,\in
\mathbb{T}\left[  \mathbf{x}\right]  ,\ \ b_{0}\neq-\infty
\end{align*}
be with simple tropical non-zero roots. Then the following two are equivalent:

\begin{enumerate}
\item $A$ and $B$ have exactly $k$ common roots.

\item $O_{\mathfrak{R}}\left(  a,b\right)  =k.$
\end{enumerate}
\end{theorem}

\begin{example}
\label{ex:result} We will illustrate the main result on a simple example. Let
\[
A=a_{0}\mathbf{x}^{3}+a_{1}\mathbf{x}^{2}+a_{2}\mathbf{x}+a_{3}\in
\mathbb{T}[\mathbf{x}]
\]
be a monic polynomial with the roots $\alpha_{1}>\alpha_{2}>\alpha_{3}%
\neq-\infty$ and let
\[
B=b_{0}\mathbf{x}^{2}+b_{1}\mathbf{x}+b_{2}\in\mathbb{T}[\mathbf{x}]
\]
be a monic polynomial with the roots $\beta_{1}>\beta_{2}\neq-\infty$. Assume
that
\[%
\begin{array}
[c]{c}%
\alpha_{1}\\
\parallel\\
\beta_{1}%
\end{array}
>\alpha_{2}>%
\begin{array}
[c]{c}%
\alpha_{3}\\
\parallel\\
\beta_{2}%
\end{array}
\neq-\infty
\]
Now we will verify the main result on this example. Note

\begin{enumerate}
\item $A$ and $B$ have exactly $2$ common roots, namely $\alpha_{1}=\beta_{1}$
and $\alpha_{3}=\beta_{2}$.

\item Table \ref{table-2} shows the value of each term in $\mathfrak{R}$.
\begin{table}[h]
\begin{center}%
\begin{tabular}
[c]{|l||l|l|}\hline
Term & Value in roots & Value simplified\\\hline\hline
$\mathbf{a}_{0}^{2}\mathbf{b}_{2}^{3}$ & $(0 )^{2}(\beta_{1}\beta_{2})^{3}$ &
$\alpha_{1}^{3}\alpha_{3}^{3}$\\\hline
$\mathbf{a}_{0}\mathbf{a}_{1}\mathbf{b}_{1}\mathbf{b}_{2}^{2}$ & $(0
)(\alpha_{1})(\beta_{1})(\beta_{1}\beta_{2})^{2}$ & $\alpha_{1}^{4}\alpha
_{3}^{2}$\\\hline
$\mathbf{a}_{0}\mathbf{a}_{2}\mathbf{b}_{0}\mathbf{b}_{2}^{2}$ & $(0
)(\alpha_{1}\alpha_{2})(0 )(\beta_{1}\beta_{2})^{2}$ & $\alpha_{1}^{3}%
\alpha_{2}\alpha_{3}^{2}$\\\hline
$\mathbf{a}_{0}\mathbf{a}_{2}\mathbf{b}_{1}^{2}\mathbf{b}_{2}$ & $(0
)(\alpha_{1}\alpha_{2})(\beta_{1})^{2}(\beta_{1}\beta_{2})$ & ${\alpha_{1}%
^{4}\alpha_{2}\alpha_{3}}$\\\hline
$\mathbf{a}_{0}\mathbf{a}_{3}\mathbf{b}_{0}\mathbf{b}_{1}\mathbf{b}_{2}$ & $(0
)(\alpha_{1}\alpha_{2}\alpha_{3})(0 )(\beta_{1})(\beta_{1}\beta_{2})$ &
$\alpha_{1}^{3}\alpha_{2}\alpha_{3}^{2}$\\\hline
$\mathbf{a}_{0}\mathbf{a}_{3}\mathbf{b}_{1}^{3}$ & $(0 )(\alpha_{1}\alpha
_{2}\alpha_{3})(\beta_{1})^{3}$ & ${\alpha_{1}^{4}\alpha_{2}\alpha_{3}}%
$\\\hline
$\mathbf{a}_{1}^{2}\mathbf{b}_{0}\mathbf{b}_{2}^{2}$ & $(\alpha_{1})^{2}(0
)(\beta_{1}\beta_{2})^{2}$ & $\alpha_{1}^{4}\alpha_{3}^{2}$\\\hline
$\mathbf{a}_{1}\mathbf{a}_{2}\mathbf{b}_{0}\mathbf{b}_{1}\mathbf{b}_{2}$ &
$(\alpha_{1})(\alpha_{1}\alpha_{2})(0 )(\beta_{1})(\beta_{1}\beta_{2})$ &
${\alpha_{1}^{4}\alpha_{2}\alpha_{3}}$\\\hline
$\mathbf{a}_{1}\mathbf{a}_{3}\mathbf{b}_{0}\mathbf{b}_{1}^{2}$ & $(\alpha
_{1})(\alpha_{1}\alpha_{2}\alpha_{3})(0 )(\beta_{1})^{2}$ & ${\alpha_{1}%
^{4}\alpha_{2}\alpha_{3}}$\\\hline
$\mathbf{a}_{1}\mathbf{a}_{3}\mathbf{b}_{0}^{2}\mathbf{b}_{2}$ & $(\alpha
_{1})(\alpha_{1}\alpha_{2}\alpha_{3})(0 )^{2}(\beta_{1}\beta_{2})$ &
$\alpha_{1}^{3}\alpha_{2}\alpha_{3}^{2}$\\\hline
$\mathbf{a}_{2}^{2}\mathbf{b}_{0}^{2}\mathbf{b}_{2}$ & $(\alpha_{1}\alpha
_{2})^{2}(0 )^{2}(\beta_{1}\beta_{2})$ & $\alpha_{1}^{3}\alpha_{2}^{2}%
\alpha_{3}$\\\hline
$\mathbf{a}_{2}\mathbf{a}_{3}\mathbf{b}_{0}^{2}\mathbf{b}_{1}$ & $(\alpha
_{1}\alpha_{2})(\alpha_{1}\alpha_{2}\alpha_{3})(0 )^{2}(\beta_{1})$ &
$\alpha_{1}^{3}\alpha_{2}^{2}\alpha_{3}$\\\hline
$\mathbf{a}_{3}^{2}\mathbf{b}_{0}^{3}$ & $(\alpha_{1}\alpha_{2}\alpha_{3}%
)^{2}(0 )^{3}$ & $\alpha_{1}^{2}\alpha_{2}^{2}\alpha_{3}^{2}$\\\hline
\end{tabular}
\end{center}
\caption{Value of each term of $\mathfrak{R}$, in Example \ref{ex:result}, at
the roots}%
\label{table-2}%
\end{table}

In the second column, we used the obvious relation
\[%
\begin{array}
[c]{cccc}%
a_{0}=0 & a_{1}=\alpha_{1} & a_{2}=\alpha_{1}\alpha_{2} & a_{3}=\alpha
_{1}\alpha_{2}\alpha_{3}\\
b_{0}=0 & b_{1}=\beta_{1} & b_{2}=\beta_{1}\beta_{2} &
\end{array}
\]
In the last column, we simplified the value using the fact that $\alpha
_{1}=\beta_{1}\ $and $\alpha_{3}=\beta_{2},$ for the sake of easier comparison
among the values. One can straightforwardly verify that $\alpha_{1}^{4}%
\alpha_{2}\alpha_{3}$ is the maximum among the values. Thus $\mathfrak{R}%
(a,b)=\alpha_{1}^{4}\alpha_{2}\alpha_{3}$.
Hence the corresponding terms are given by%
\[
E_{\mathfrak{R}}(a,b)=\{\mathbf{a}_{0}\mathbf{a}_{2}\mathbf{b}_{1}%
^{2}\mathbf{b}_{2}\mathbf{,\;\mathbf{a}_{0}\mathbf{a}_{3}\mathbf{b}_{1}%
^{3},\;\mathbf{a}_{1}\mathbf{a}_{2}\mathbf{b}_{0}\mathbf{b}_{1}\mathbf{b}%
_{2},\;a}_{1}\mathbf{a}_{3}\mathbf{b}_{0}\mathbf{b}_{1}^{2}\}
\]
Thus $\# E_{\mathfrak{R}}\left(  a,b\right)  =2^{2}$ and $O_{\mathfrak{R}%
}\left(  a,b\right)  =2$.
\end{enumerate}

\noindent We have verified the main result on this example.
\end{example}

\section{Proof}
\label{sec:proof}
{
In this section, we provide a proof the main result (Theorem~\ref{thm:main}).
One naturally wonders whether a proof for the main result
can be obtained by suitably `translating'' a proof for the field case (such as the one  given in Appendix).
We have tried the approach, without success. Thus we developed a completely different proof strategy.

Before plunging into a technically detailed proof, we first provide an
informal overview of the   proof strategy.
Note that the tropical resultant $\mathfrak{R}$ is defined as the tropicalization of
(i.e., the tropical sum of all the terms appearing in) the resultant over $\mathbb{C}$.
Recall that the resultant over $\mathbb{C}$ is same as  the determinant of the Sylvester matrix.
Hence the tropical resultant $\mathfrak{R}$ is
the tropicalization of the determinant of the Sylvester matrix. Recalling that
terms in the determinant corresponds to permutations of column indices, we
observe that each term in $\mathfrak{R}$ comes from one or more permutations
of $(1,\ldots,n+m)$. Thus we focus our attention to the permutations. Let $S$
be the set of all the permutations. Then, the main steps of the proof consist
of the followings.

\begin{enumerate}
\item Lemma \ref{lem:no_zigzag}: We \textquotedblleft prune\textquotedblright%
\ the set $S$, obtaining $S^{\ast}$, by removing all the permutations that
never yields $\mathfrak{R}(a,b)$, no matter what $a$ and $b$ are.

\item Lemma \ref{lem:unique}: We show that each permutation in $S^{\ast}$
provides a different term in $\mathfrak{R}(a,b)$.

\item Lemma \ref{lem:RPPstar}: Using the above two lemmas,
we show that the following three elements of $\mathbb{T}$
are the same: $\mathfrak{R}(a,b)$,
$P$ and~$P^{\ast}$,
where the last two are obtained
from the Sylvester matrix, by considering
all the permutations in $S$ and $S^{\ast}$, respectively.

\item Lemma \ref{lem:sort}: We \textquotedblleft
characterize\textquotedblright\ the permutations in $S^{\ast}$ that yield
$\mathfrak{R}(a,b)$, in terms of the ordering among the roots of $A$ and $B$.

\item Lemma \ref{lem:ncr<->nmp}: We show that the number of permutation in the
previous step is exactly $2^{k}$, where $k$ is the number of common roots.

\item The main result is immediate from the above lemmas.
\end{enumerate}

\bigskip

\noindent Now we plunge into the details of the proof. From now on, we fix
$a=(a_{0},\ldots,a_{m})\in\mathbb{T}^{m+1}$ and $b=(b_{0},\ldots,b_{n}%
)\in\mathbb{T}^{n+1}$ such that the polynomials
\begin{align*}
A  &  =a_{0}\mathbf{x}^{m}+a_{1}\mathbf{x}^{m-1}+\cdots+a_{m}\in
\mathbb{T}\left[  \mathbf{x}\right]  , \,\,a_{0} \neq-\infty\\
B  &  =b_{0}\mathbf{x}^{n}+b_{1}\mathbf{x}^{n-1}+\cdots+b_{n}\;\;\;\in
\mathbb{T}\left[  \mathbf{x}\right]  ,\,\,b_{0} \neq-\infty
\end{align*}
are with simple tropical non-zero roots.

The next two lemmas will be used to reduced the proof to the monic polynomial case.

\begin{lemma}
\label{lem:monic1} Let $C=c_{0}\mathbf{x}^{d}+c_{1}\mathbf{x}^{d-1}%
+\cdots+c_{d}\in\mathbb{T}\left[  \mathbf{x}\right]  $ where $c_{0}\neq
-\infty$. The roots of $C$ and $\frac{1}{c_{0}}C$ are the same.
\end{lemma}

\begin{proof}
Obvious.
\end{proof}

\begin{lemma}
\label{lem:monic} $O_{\mathfrak{R}}\left(  a,b\right)  = O_{\mathfrak{R}%
}\left(  \frac{1}{a_{0}}a,\frac{1}{b_{0}}b\right)  .$
\end{lemma}

\begin{proof}
Let $\mathfrak{R}$ be expressed as $\mathfrak{R}=t_{1}+\cdots+t_{\rho}$ where
$t_{i}$ are terms in $\mathbf{a},\mathbf{b}$ with tropical non-zero
coefficients. We observe that $\mathfrak{R}$ is bi-homogeneous of degrees $n$
and $m$ in the variables $\mathbf{a}$ and $\mathbf{b}$, respectively. Thus for
each term $t_{i}$ we have
\[
a_{0}^{m}b_{0}^{n}\;\;t_{i}\left(  \frac{1}{a_{0}}a,\frac{1}{b_{0}}b\right)
\;\;=\;\; t_{i}(a,b).
\]
Now, the lemma follows immediately from the definition of order.
\end{proof}

Therefore, taking into account of Lemma \ref{lem:monic1} and Lemma
\ref{lem:monic}, we can restrict the proof of the main result to the monic
case, without losing generality. Thus, in the following, we assume that $A, B$
are monic; that is $a_{0}=0 =b_{0}$. In addition, let $\alpha_{1}%
>\cdots>\alpha_{m}$ be the roots of $A$ and $\beta_{1}>\cdots>\beta_{n}$ be
the roots of $B.$ We obviously have
\begin{align*}
a_{i}  &  =\alpha_{1}\cdots\alpha_{i}\\
b_{i}  &  =\beta_{1}\cdots\beta_{i}%
\end{align*}
Note that, since $\alpha_{i}\neq-\infty$ and $\beta_{j}\neq-\infty$, we see
that $a_{i}\neq-\infty$ and $b_{j}\neq-\infty$. We will set $a_{i}=-\infty$ if
$i>m$ or $i<0$ and set $b_{i}=-\infty$ if $i>n$ or $i<0.$ \ Let%
\begin{align*}
&  \left[  a_{j-i}\right]  =%
\begin{bmatrix}
a_{0} & a_{1} & \cdots & \cdots & a_{m} &  & \\
& \ddots & \ddots &  &  & \ddots & \\
&  & a_{0} & a_{1} & \cdots & \cdots & a_{m}%
\end{bmatrix}
\in\mathbb{T}^{n\times\left(  n+m\right)  }\\
&  \left[  b_{j-i}\right]  =%
\begin{bmatrix}
b_{0} & b_{1} & \cdots & b_{n} &  &  & \\
& \ddots & \ddots &  & \ddots &  & \\
&  & \ddots & \ddots &  & \ddots & \\
&  &  & b_{0} & b_{1} & \cdots & b_{n}%
\end{bmatrix}
\ \ \ \in\mathbb{T}^{m\times\left(  n+m\right)  }\\
M  &  :=\left[
\begin{array}
[c]{c}%
\left[  a_{j-i}\right] \\
\left[  b_{j-i}\right]
\end{array}
\right]  \in\mathbb{T}^{\left(  n+m\right)  \times\left(  n+m\right)  }%
\end{align*}

\begin{notation}
Let $S$ stand for the set of all permutations of $\left(  1,\ldots,n+m\right)
$. Furthermore, for $\pi\in S$, let
\[
M_{\pi}:=M_{1,\pi_{1}}\cdots M_{n+m,\pi_{n+m}}.
\]
Moreover, let $S^{\ast}$ stand for the set of all $\left(  \nu_{1},\ldots
,\nu_{n},\mu_{1},\ldots,\mu_{m}\right)  \in S$ such that $\nu_{1}<\cdots
<\nu_{n}$ and $\mu_{1}<\cdots<\mu_{m}.$
\end{notation}

\begin{remark}
From the structure of $M$ it follows that, if $\pi=\left(  \nu_{1},\ldots
,\nu_{n},\mu_{1},\ldots,\mu_{m}\right)  \in S,$ then
\[
M_{\pi}=a_{\nu_{1}-1}\cdots a_{\nu_{n}-n}\ b_{\mu_{1}-1}\cdots b_{\mu_{m}-m}%
\]

\end{remark}

In the next example we see that $S^{*}$ is much smaller that $S$. Later, we
will see that the relevant information for our problem lies in $S^{*}$

\begin{example}
Let $m=3$ and $n=2$. Then
\begin{align*}
S^{\ast}=\{  &
\ (1,2,\ \ 3,4,5),\ \ \ (1,3,\ \ 2,4,5),\ \ \ (1,4,\ \ 2,3,5),\ \ \ (1,5,\ \ 2,3,4),\ \ \\
&  \ (2,3,\ \ 1,4,5),\ \ \ (2,4,\ \ 1,3,5),\ \ \ (2,5,\ \ 1,3,4),\ \ \\
&  \ (3,4,\ \ 1,2,5),\ \ \ (3,5,\ \ 1,2,4),\ \ \\
&  \ (4,5,\ \ 1,2,3)\ \}
\end{align*}
\newline Note
\[
\#S^{\ast}=\frac{(n+m)!}{n!m!}=\frac{(2+3)!}{2!3!}=10, \,\,\, \#S=120.
\]

\end{example}

Next, we introduce the following two elements
\[
P=\sum_{\pi\in S}M_{\pi},\,\,\text{and}\,\,P^{*}=\sum_{\pi\in S^{*}}M_{\pi}.
\]
The next lemmas will be used to conclude that $P=P^{*}=\mathfrak{R}(a,b)$ (see
Lemma~\ref{lem:RPPstar}).

\begin{lemma}
\label{lemma:DversusP} $\mathfrak{R}(a,b)\leq P. $
\end{lemma}

\begin{proof}
Obvious from the notion of tropicalization.
\end{proof}

\begin{lemma}
[Odagiri 2008, \cite{Od08}]\label{lem:zigzag}Let $\pi=\left(  \nu_{1}%
,\ldots,\nu_{n},\mu_{1},\ldots,\mu_{m}\right)  \in S$ be such that $M_{\pi
}\neq-\infty.$ Then we have the followings.

\begin{enumerate}
\item Suppose that $\nu_{k}>\nu_{k+1}$ for some $k$. Let $\pi^{\prime
}:=\left(  \nu_{1},\ldots,\nu_{k+1},\nu_{k},\ldots,\nu_{n},\mu_{1},\ldots
,\mu_{m}\right)  ,$ that is, obtained from~$\pi$ by swapping $\nu_{k}$ and
$\nu_{k+1}.$ Then $M_{\pi^{\prime}}>M_{\pi}.$

\item Suppose that $\mu_{k}>\mu_{k+1}$ for some $k$. Let $\pi^{\prime
}:=\left(  \nu_{1},\ldots,\nu_{n},\mu_{1},\ldots,\mu_{k+1},\mu_{k},\ldots
,\mu_{m}\right)  ,$ that is, obtained from~$\pi$ by swapping $\mu_{k}$ and
$\mu_{k+1}.$ Then $M_{\pi^{\prime}}>M_{\pi}.$
\end{enumerate}
\end{lemma}

\begin{example}
Let $m=3$ and $n=3.$ Let
\[
\pi=(1,4,3,2,5)\ \ \ \ \text{and}\ \ \ \ \ \pi^{\prime}=(1,4,2,3,5)
\]
They represent the following choices of elements (\textquotedblleft
path\textquotedblright) encircled
\[
\pi:%
\begin{bmatrix}
\tikz[baseline=(X.base)] \node (X) [draw, shape=circle, inner sep=0] {\strut
$a_0$}; & a_{1} & a_{2} & a_{3} & \\
& a_{0} & a_{1} & \tikz[baseline=(X.base)] \node (X) [draw, shape=circle,
inner sep=0] {\strut$a_2$}; & a_{3}\\
b_{0} & b_{1} & \tikz[baseline=(X.base)] \node (X) [draw, shape=circle, inner
sep=0] {\strut$b_2$}; &  & \\
& \tikz[baseline=(X.base)] \node (X) [draw, shape=circle, inner sep=0]
{\strut$b_0$}; & b_{1} & b_{2} & \\
&  & b_{0} & b_{1} & \tikz[baseline=(X.base)] \node (X) [draw, shape=circle,
inner sep=0] {\strut$b_2$};
\end{bmatrix}
\ \ \ \ \ \,\text{and}\ \ \ \ \ \pi^{\prime}:%
\begin{bmatrix}
\tikz[baseline=(X.base)] \node (X) [draw, shape=circle, inner sep=0] {\strut
$a_0$}; & a_{1} & a_{2} & a_{3} & \\
& a_{0} & a_{1} & \tikz[baseline=(X.base)] \node (X) [draw, shape=circle,
inner sep=0] {\strut$a_2$}; & a_{3}\\
b_{0} & \tikz[baseline=(X.base)] \node (X) [draw, shape=circle, inner sep=0]
{\strut$b_1$}; & b_{2} &  & \\
& b_{0} & \tikz[baseline=(X.base)] \node (X) [draw, shape=circle, inner sep=0]
{\strut$b_1$}; & b_{2} & \\
&  & b_{0} & b_{1} & \tikz[baseline=(X.base)] \node (X) [draw, shape=circle,
inner sep=0] {\strut$b_2$};
\end{bmatrix}
\]
Note
\[
\frac{M_{\pi^{\prime}}}{M_{\pi}}=\frac{a_{0}a_{2}b_{1}b_{1}b_{2}}{a_{0}%
a_{2}b_{2}b_{0}b_{2}}=\frac{b_{1}b_{1}}{b_{2}b_{0}}=\frac{\left(  \beta
_{1}\right)  \left(  \beta_{1}\right)  }{\left(  \beta_{1}\beta_{2}\right)
\left(  0 \right)  }=\frac{\beta_{1}}{\beta_{2}}>0
\]
Thus%
\[
M_{\pi^{\prime}}>M_{\pi}%
\]
verifying the lemma on $\pi$ and $\pi^{\prime}$. Observe that $\pi$ has a
\textquotedblleft zigzag\textquotedblright\ in the bottom part, while
$\pi^{\prime}$ does not have a zigzag. The lemma says that a zigzag makes the
value of a path smaller.
\end{example}

\begin{proof}
[Proof of Lemma \ref{lem:zigzag}]The proof was given in \cite{Od08}. However
for the sake of reader's convenience and the notational consistency, we
provide a complete proof here. We will show the proof of the claim $1$ only.
The proof for the claim~$2$ is essentially the same. Note%
\begin{align*}
\frac{M_{\pi^{\prime}}}{M_{\pi}}  &  =\frac{a_{\nu_{1}-1}\cdots a_{\nu
_{k+1}-k}\ a_{\nu_{k}-\left(  k+1\right)  }\cdots a_{\nu_{n}-n}\ b_{\mu_{1}%
-1}\cdots b_{\mu_{m}-m}}{a_{\nu_{1}-1}\cdots a_{\nu_{k}-k}\ a_{\nu
_{k+1}-\left(  k+1\right)  }\cdots a_{\nu_{n}-n}\ b_{\mu_{1}-1}\cdots
b_{\mu_{m}-m}}\\
&  =\frac{a_{\nu_{k+1}-k}\ \ a_{\nu_{k}-\left(  k+1\right)  }}{a_{\nu_{k}%
-k}\ \ a_{\nu_{k+1}-\left(  k+1\right)  }}\\
&  =\frac{a_{\nu_{k+1}-\left(  k+1\right)  }\ \alpha_{\nu_{k+1}-k}%
\ \ a_{\nu_{k}-\left(  k+1\right)  }}{a_{\nu_{k}-\left(  k+1\right)  }%
\ \alpha_{\nu_{k}-k}\ \ a_{\nu_{k+1}-\left(  k+1\right)  }}\\
&  =\frac{\alpha_{\nu_{k+1}-k}}{\alpha_{\nu_{k}-k}}\\
&  >0
\end{align*}
Thus $M_{\pi^{\prime}}>M_{\pi}.$
\end{proof}

\begin{lemma}
\label{lem:no_zigzag}If $\pi\in S\setminus S^{\ast},$ then $P>M_{\pi}.$
\end{lemma}

\begin{proof}
Let $\pi\in S\setminus S^{\ast}.$ Recall that $P=\prod_{ij}\left(  \alpha
_{i}+\beta_{j}\right)  .$ Since $\alpha_{1},\ldots,\alpha_{m},\beta_{1}%
,\ldots,\beta_{n}$ are tropically non-zero, we have $P\neq$ $-\infty$ Suppose
that $M_{\pi}=-\infty.$ Then we obviously have $P>M_{\pi}.$ Thus from now on,
assume that $M_{\pi}\neq-\infty.$ Note
\[
P=\sum_{\rho\in S}M_{\rho}=\sum_{\rho\in S\setminus\left\{  \pi\right\}
}M_{\rho}+M_{\pi}%
\]
From Lemma \ref{lem:zigzag}, we have $\pi^{\prime}\in S\setminus\left\{
\pi\right\}  $ such that $M_{\pi^{\prime}}>M_{\pi}.$ Hence $P>M_{\pi}.$
\end{proof}

\begin{lemma}
\label{lem:s*} $P=P^{*}.$
\end{lemma}

\begin{proof}
Immediate from the definition of $P$, $P^{*}$. and Lemma \ref{lem:no_zigzag}.
\end{proof}

\begin{lemma}
\label{lem:unique}Let $\pi,\pi^{\prime}\in S^{\ast}.$ The following two are equivalent.

\begin{enumerate}
\item $\pi=\pi^{\prime}$

\item ${M}_{\pi}={M}_{\pi^{\prime}}$
\end{enumerate}
\end{lemma}

\begin{proof}
Let $\pi=\left(  \nu,\mu\right)  ,\ \pi^{\prime}=\left(  \nu^{\prime}%
,\mu^{\prime}\right)  \in S^{\ast}$. It is obvious that $1\Longrightarrow2.$
It remains to show $2\Longrightarrow1.$ Assume ${M}_{\pi}={M}_{\pi^{\prime}}.$
We will show that $\pi=\pi^{\prime}.$ Note%
\begin{align*}
{M}_{\pi}  &  ={a}_{\bar{\nu}_{1}}\cdots{a}_{\bar{\nu}_{n}} {b}_{\bar{\mu}%
_{1}}\cdots{b}_{\bar{\mu}_{m}}\\
{M}_{\pi^{\prime}}  &  = {a}_{\bar{\nu}_{1}^{\prime}}\cdots{a}_{\bar{\nu}%
_{n}^{\prime}} {b}_{\bar{\mu}_{1}^{\prime}}\cdots{b}_{\bar{\mu}_{m}^{\prime}}%
\end{align*}
where $\bar{\mu}_{i}=\mu_{i}-i,$ $\bar{\nu}_{i}=\nu_{i}-i$, $\bar{\mu}%
_{i}^{\prime}=\mu_{i}^{\prime}-i$ and $\bar{\nu}_{i}^{\prime}=\nu_{i}^{\prime
}-i.$ Since $\pi\in S^{\ast}\ $we have $0\leq\bar{\nu}_{1}\leq\cdots\leq
\bar{\nu}_{n}\leq m\ $and $0\leq\bar{\mu}_{1}\leq\ldots\leq\bar{\mu}_{m}\leq
n.$ The same with $\left(  \bar{\nu}^{\prime},\bar{\mu}^{\prime}\right)  .$
Since ${M}_{\pi}={M}_{\pi^{\prime}}$, we have $\bar{\nu}=\bar{\nu}^{\prime}$
and $\bar{\mu}=\bar{\mu}^{\prime}$, in turn, $\nu=\nu^{\prime}$ and $\mu
=\mu^{\prime}.$ Thus $\pi=\pi^{\prime}.$
\end{proof}

\begin{lemma}
\label{lem:RPPstar} $\mathfrak{R}(a,b)=P=P^{*}$.
\end{lemma}

\begin{proof}
By Lemmas \ref{lem:s*} and \ref{lem:unique} we get that $P\leq\mathfrak{R}%
(a,b)$, and by Lemma \ref{lemma:DversusP} we know that $\mathfrak{R}(a,b)\leq
P$. So, the statement holds.
\end{proof}

\begin{lemma}
\label{lem:pq} Let $\pi=\left(  \nu_{1},\ldots,\nu_{n},\mu_{1},\ldots,\mu
_{m}\right)  \in S^{\ast}.$ Let $p_{i}=n-\left(  \mu_{i}-i\right)  $ and
$q_{i}=m-\left(  \nu_{i}-i\right)  .$ Then we have%
\[
M_{\pi}=\alpha^{p}\beta^{q}%
\]

\end{lemma}

\begin{example}
\label{ex:pq} Let us $m=3$ and $n=2.$ Let $\pi=(1,4,\ 2,3,5)\in S^{\ast}.$
Then%
\begin{align*}
M_{\pi}  &  =a_{1-1}\ a_{4-2}\ \ b_{2-1}\ b_{3-2}\ b_{5-3}\\
&  =a_{0}a_{2}b_{1}b_{1}b_{2}\\
&  =(0 )(\alpha_{1}\alpha_{2})(\beta_{1})(\beta_{1})(\beta_{1}\beta_{2})\\
&  =\alpha_{1}^{1}\alpha_{2}^{1}\alpha_{3}^{0}\beta_{1}^{3}\beta_{2}^{1}\\
\alpha^{p}\beta^{q}  &  =\alpha_{1}^{2-\left(  2-1\right)  }\alpha
_{2}^{2-\left(  3-2\right)  }\alpha_{3}^{2-\left(  5-3\right)  }\beta
_{1}^{3-\left(  1-1\right)  }\beta_{2}^{3-\left(  4-2\right)  }\\
&  =\alpha_{1}^{1}\alpha_{2}^{1}\alpha_{3}^{0}\beta_{1}^{3}\beta_{2}^{1}%
\end{align*}
Hence $M_{\pi}=\alpha^{p}\beta^{q},$ verifying the lemma on the particular
$\pi.$
\end{example}

\begin{proof}
[Proof of Lemma \ref{lem:pq}]There are two cases: $\mu_{1}=1$ or $\nu_{1}=1.$
We will show a proof of the lemma only for the case $\mu_{1}=1$. The proof for
the case $\nu_{1}=1$ is essentially the same. We will divide the proof into
two steps.

\begin{enumerate}
\item Let $s_{1},s_{2},\ldots$ be the lengths of the consecutive blocks in
$\mu$. Likewise let $t_{1},t_{2},\ldots$ be the lengths of the consecutive
blocks in $\nu.$ Then%
\begin{align*}
\mu &  =\left(  1,\ldots,s_{1},\ \ \ s_{1}+t_{1}+1,\ldots,s_{1}+t_{1}%
+s_{2},\ \ \ s_{1}+t_{1}+s_{2}+t_{2}+1,\ldots,s_{1}+t_{1}+s_{2}+t_{2}%
+s_{3},\ \ \ \ldots\right) \\
\nu &  =\left(  s_{1}+1,\ldots,s_{1}+t_{1},\ \ \ s_{1}+t_{1}+s_{2}%
+1,\ldots,s_{1}+t_{1}+s_{2}+t_{2},\ \ \ \ldots\right)
\end{align*}
Let $\overline{\mu}_{j}=\mu_{j}-j$ and $\overline{\nu}_{i}=\nu_{i}-i$. Then
\begin{align}
\bar{\mu}  &  =(\underbrace{0,\ldots,0}_{s_{1}},\ \ \ \underbrace{t_{1}%
,\ldots,t_{1}}_{s_{2}},\ \ \ \underbrace{t_{1}+t_{2},\ldots,t_{1}+t_{2}%
}_{s_{3}},\ \ \ \ldots)\label{st1}\\
\bar{\nu}  &  =(\underbrace{s_{1},\ldots,s_{1}}_{t_{1}}%
,\ \ \ \underbrace{s_{1}+s_{2},\ldots,s_{1}+s_{2}}_{t_{2}},\ \ \ \ldots
)\nonumber
\end{align}

\item Note%
\begin{align*}
M_{\pi}  &  =a_{\overline{\nu}_{1}}a_{\overline{\nu}_{2}}\cdots b_{\overline
{\mu}_{1}}b_{\overline{\mu}_{2}}\cdots\\
&  =a_{s_{1}}^{t_{1}}a_{s_{1}+s_{2}}^{t_{2}}\cdots b_{t_{1}}^{s_{2}}%
b_{t_{1}+t_{2}}^{s_{3}}\cdots\ \ \ \ \ \text{from (\ref{st1}) and the fact
that }b_{0}=1\text{{}}\\
&  =\prod_{k=1}^{s_{1}}\alpha_{k}^{t_{1}}\prod_{k=1}^{s_{1}+s_{2}}\alpha
_{k}^{t_{2}}\cdots\prod_{k=1}^{t_{1}}\beta_{k}^{s_{2}}\prod_{k=1}^{t_{1}%
+t_{2}}\beta_{k}^{s_{3}}\cdots\\
&  =\prod_{k=1}^{s_{1}}\alpha_{k}^{t_{1}+t_{2}+\cdots}\prod_{k=s_{1}+1}%
^{s_{1}+s_{2}}\alpha_{k}^{t_{2}+t_{3}+\cdots}\cdots\prod_{k=1}^{t_{1}}%
\beta_{k}^{s_{2}+s_{3}+\cdots}\prod_{k=t_{1}+1}^{t_{1}+t_{2}}\beta_{k}%
^{s_{3}+s_{4}+\cdots}\cdots\\
&  =\prod_{k=1}^{s_{1}}\alpha_{k}^{n-0}\prod_{k=s_{1}+1}^{s_{1}+s_{2}}%
\alpha_{k}^{n-t_{1}}\cdots\prod_{k=1}^{t_{1}}\beta_{k}^{m-s_{1}}\prod
_{k=t_{1}+1}^{t_{1}+t_{2}}\beta_{k}^{m-\left(  s_{1}+s_{2}\right)  }\cdots\\
&  =\prod_{k=1}^{s_{1}}\alpha_{k}^{n-\bar{\mu}_{k}}\prod_{k=s_{1}+1}%
^{s_{1}+s_{2}}\alpha_{k}^{n-\bar{\mu}_{k}}\cdots\prod_{k=1}^{t_{1}}\beta
_{k}^{m-\bar{\nu}_{k}}\prod_{k=t_{1}+1}^{t_{1}+t_{2}}\beta_{k}^{m-\bar{\nu
}_{k}}\cdots\\
&  =\prod_{k=1}^{m}\alpha_{k}^{n-\bar{\mu}_{k}}\prod_{k=1}^{n}\beta
_{k}^{m-\bar{\nu}_{k}}\\
&  =\prod_{k=1}^{m}\alpha_{k}^{p_{k}}\prod_{k=1}^{n}\beta_{k}^{q_{k}}\\
&  =\alpha^{p}\beta^{q}%
\end{align*}

\end{enumerate}
\end{proof}

\begin{notation}
\label{not:gamma} Let $\pi\in S$ and $L$ be a list of length $n+m$. Then
$\pi\left(  L\right)  $ is the list obtained from $L$ by permuting the element
according to $\pi$, that is, by moving the $i$-th element of $L$ to the
$\pi_{i}$-th position.
\end{notation}

\begin{example}
Let $m=3$ and $n=2.$ Let $\pi=(1,4,2,3,5)\in S^{\ast}$. Then%
\[
\pi\left(  \beta_{1},\beta_{2},\alpha_{1},\alpha_{2},\alpha_{3}\right)
=\left(  \beta_{1},\alpha_{1},\alpha_{2},\beta_{2},\alpha_{3}\right)
\]

\end{example}

\begin{lemma}
\label{lem:sort}Let $\pi\in S$. The followings are equivalent

\begin{enumerate}
\item The elements of $\pi\left(  \beta,\alpha\right)  $ are non-increasing

\item $P=M_{\pi}.$
\end{enumerate}
\end{lemma}

\begin{example}
\label{ex:sort} We illustrate the lemma using Example \ref{ex:result}, that
is, $m=3$ and $n=2$ and
\[%
\begin{array}
[c]{c}%
\alpha_{1}\\
\parallel\\
\beta_{1}%
\end{array}
>\alpha_{2}>%
\begin{array}
[c]{c}%
\alpha_{3}\\
\parallel\\
\beta_{2}%
\end{array}
\neq-\infty
\]
For $\pi\in S\backslash S^{\ast},$ it is easy to check that both $1$ and $2$
are false and thus the lemma is verified. For $\pi\in S^{\ast}$, Table
\ref{table-3} verifies the lemma.
\begin{table}[h]
\begin{center}%
\begin{tabular}
[c]{|c||c|c|c||c|c|c|}\hline
$\pi\in S^{\ast}$ & $\gamma$ & $\gamma\text{ simplified}$ & $\gamma_{1}%
\geq\cdots\geq\gamma_{5}$ & $M_{\pi}$ & $M_{\pi} \text{simplified}$ &
$P=M_{\pi}$\\\hline
$(1,2,\ \ 3,4,5)$ & $\left(  \beta_{1},\beta_{2},\alpha_{1},\alpha_{2}%
,\alpha_{3}\right)  $ & $\left(  \alpha_{1},\alpha_{3},\alpha_{1},\alpha
_{2},\alpha_{3}\right)  $ & false & $\beta_{1}^{3}\beta_{2}^{3}$ & $\alpha
_{1}^{3}\alpha_{3}^{3}$ & false\\\hline
$(1,3,\ \ 2,4,5)$ & $\left(  \beta_{1},\alpha_{1},\beta_{2},\alpha_{2}%
,\alpha_{3}\right)  $ & $\left(  \alpha_{1},\alpha_{1},\alpha_{3},\alpha
_{2},\alpha_{3}\right)  $ & false & $\alpha_{1}\beta_{1}^{3}\beta_{2}^{2}$ &
$\alpha_{1}^{4}\alpha_{3}^{2}$ & false\\\hline
$\ (1,4,\ \ 2,3,5)$ & $\left(  \beta_{1},\alpha_{1},\alpha_{2},\beta
_{2},\alpha_{3}\right)  $ & $\left(  \alpha_{1},\alpha_{1},\alpha_{2}%
,\alpha_{3},\alpha_{3}\right)  $ & {true} & $\alpha_{1}\alpha_{2}\beta_{1}%
^{3}\beta_{2}$ & $\alpha_{1}^{4}\alpha_{2}\alpha_{3}$ & {true}\\\hline
$\ (1,5,\ \ 2,3,4)$ & $\left(  \beta_{1},\alpha_{1},\alpha_{2},\alpha
_{3},\beta_{2}\right)  $ & $\left(  \alpha_{1},\alpha_{1},\alpha_{2}%
,\alpha_{3},\alpha_{3}\right)  $ & {true} & $\alpha_{1}\alpha_{2}\alpha
_{3}\beta_{1}^{3}$ & $\alpha_{1}^{4}\alpha_{2}\alpha_{3}$ & {true}\\\hline
$(2,3,\ \ 1,4,5)$ & $\left(  \alpha_{1},\beta_{1},\beta_{2}.\alpha_{2}%
,\alpha_{3}\right)  $ & $\left(  \alpha_{1},\alpha_{1},\alpha_{3}.\alpha
_{2},\alpha_{3}\right)  $ & \multicolumn{1}{||c||}{false} & $\alpha_{1}%
^{2}\beta_{1}^{2}\beta_{2}^{2}$ & $\alpha_{1}^{4}\alpha_{3}^{2}$ &
false\\\hline
$(2,4,\ \ 1,3,5)$ & $\left(  \alpha_{1},\beta_{1},\alpha_{2},\beta_{2}%
,\alpha_{3}\right)  $ & $\left(  \alpha_{1},\alpha_{1},\alpha_{2},\alpha
_{3},\alpha_{3}\right)  $ & {true} & $\alpha_{1}^{2}\alpha_{2}\beta_{1}%
^{2}\beta_{2}$ & $\alpha_{1}^{4}\alpha_{2}\alpha_{3}$ & {true}\\\hline
$\ (2,5,\ \ 1,3,4)$ & $\left(  \alpha_{1},\beta_{1},\alpha_{2},\alpha
_{3},\beta_{2}\right)  $ & $\left(  \alpha_{1},\alpha_{1},\alpha_{2}%
,\alpha_{3},\alpha_{3}\right)  $ & {true} & $\alpha_{1}^{2}\alpha_{2}\beta
_{1}^{2}\alpha_{3}$ & $\alpha_{1}^{4}\alpha_{2}\alpha_{3}$ & {true}\\\hline
$(3,4,\ \ 1,2,5)$ & $\left(  \alpha_{1},\alpha_{2},\beta_{1},\beta_{2}%
,\alpha_{3}\right)  $ & $\left(  \alpha_{1},\alpha_{2},\alpha_{1},\alpha
_{3},\alpha_{3}\right)  $ & false & $\alpha_{1}^{2}\alpha_{2}^{2}\beta
_{1}\beta_{2}$ & $\alpha_{1}^{3}\alpha_{2}^{2}\alpha_{3}$ & false\\\hline
$(3,5,\ \ 1,2,4)$ & $\left(  \alpha_{1},\alpha_{2},\beta_{1},\alpha_{3}%
.\beta_{2}\right)  $ & $\left(  \alpha_{1},\alpha_{2},\alpha_{1},\alpha
_{3}.\alpha_{3}\right)  $ & false & $\alpha_{1}^{2}\alpha_{2}^{2}\beta
_{1}\alpha_{3}$ & $\alpha_{1}^{3}\alpha_{2}^{2}\alpha_{3}$ & false\\\hline
$\ (4,5,\ \ 1,2,3)$ & $\left(  \alpha_{1},\alpha_{2},\alpha_{3},\beta
_{1},\beta_{2}\right)  $ & $\left(  \alpha_{1},\alpha_{2},\alpha_{3}%
,\alpha_{1},\alpha_{3}\right)  $ & false & $\alpha_{1}^{2}\alpha_{2}^{2}%
\alpha_{3}^{2}$ & $\alpha_{1}^{2}\alpha_{2}^{2}\alpha_{3}^{2}$ & false\\\hline
\end{tabular}
\end{center}
\caption{Verification of Lemma \ref{lem:sort} in Example \ref{ex:sort}, where
$\gamma=\pi(\beta,\alpha)$}%
\label{table-3}%
\end{table}
In the table, $\gamma$ stands for $\pi(\beta,\alpha)$. In the 3rd and the 6th
columns, we simplified the previous columns using the fact that $\alpha
_{1}=\beta_{1}\ $and $\alpha_{3}=\beta_{2},$ for the sake of easier checks in
the next columns.
\end{example}

\begin{proof}
[Proof of Lemma \ref{lem:sort}]Let $\gamma=\pi\left(  \beta,\alpha\right)  $.
We divide the proof into two cases: $\pi\in S\setminus S^{\ast}$ and $\pi\in
S^{\ast}.$

\medskip

\noindent Case 1: $\pi\in S\setminus S^{\ast}.$ For some $i<j,$ we have
$\nu_{i}>\nu_{j}$ or $\mu_{i}>\mu_{j}.$ Recall that $\gamma_{\nu_{i}}%
=\beta_{i},$ $\gamma_{\nu_{j}}=\beta_{j},$ $\gamma_{\mu_{i}}=\alpha_{i}$
and~$\gamma_{\mu_{j}}=\alpha_{j}.$ Since $\beta_{i}>\beta_{j}$ and $\alpha
_{i}>\alpha_{j},$ we have $\gamma_{\nu_{i}}>\gamma_{\nu_{j}}$ or $\gamma
_{\mu_{i}}>\gamma_{\mu_{j}}.$ Thus the statement $\gamma_{1}\geq\gamma_{2}%
\geq\cdots\geq\gamma_{n+m}$ is false. From Lemma \ref{lem:no_zigzag}, the
statement $P=M_{\pi}$ is also false. Thus the lemma is vacuously true.

\medskip

\noindent Case 2: $\pi\in S^{\ast}.$ We prove each direction of implication
one at a time:

\begin{enumerate}
\item If $\gamma_{1}\geq\gamma_{2}\geq\cdots\geq\gamma_{n+m}$ then $P=M_{\pi
}.$

Assume that $\gamma_{1}\geq\gamma_{2}\geq\cdots\geq\gamma_{n+m}.$ We need to
show $P=M_{\pi}.$ By Lemma \ref{lem:s*}%
\[
P=\sum_{\pi^{\prime}\in S^{\ast}}M_{\pi^{\prime}}%
\]
Let $\pi^{\prime}=\left(  \nu_{1}^{\prime},\ldots,\nu_{n}^{\prime},\mu
_{1}^{\prime},\ldots,\mu_{m}^{\prime}\right)  \in S^{\ast}$ be such that
$\pi^{\prime}\neq\pi.$ It suffices to show that $M_{\pi}\geq M_{\pi^{\prime}%
}.$ If $M_{\pi^{\prime}}=-\infty$ then it is obvious true. Thus from now on,
assume that $M_{\pi^{\prime}}\neq-\infty.$ Note, by Lemma \ref{lem:pq},%
\begin{align*}
\frac{M_{\pi}}{M_{\pi^{\prime}}}  &  =\frac{\prod_{i=1}^{m}\alpha
_{i}^{n-\left(  \mu_{i}-i\right)  }\prod_{i=1}^{n}\beta_{i}^{m-\left(  \nu
_{i}-i\right)  }}{\prod_{i=1}^{m}\alpha_{i}^{n-\left(  \mu_{i}^{\prime
}-i\right)  }\prod_{i=1}^{n}\beta_{i}^{m-\left(  \nu_{i}^{\prime}-i\right)  }%
}\\
&  =\frac{\prod_{i=1}^{m}\alpha_{i}^{\mu_{i}\prime}\prod_{i=1}^{n}\beta
_{i}^{\nu_{i}^{\prime}}}{\prod_{i=1}^{m}\alpha_{i}^{\mu_{i}}\prod_{i=1}%
^{n}\beta_{i}^{\nu_{i}}}\\
&  =\frac{\prod_{i=1}^{m}\gamma_{\mu_{i}}^{\mu_{i}^{\prime}}\prod_{i=1}%
^{n}\gamma_{\nu_{i}}^{\nu_{i}^{\prime}}}{\prod_{i=1}^{m}\gamma_{\mu_{i}}%
^{\mu_{i}}\prod_{i=1}^{n}\gamma_{\nu_{i}}^{\nu_{i}}}\\
&  =\frac{\prod_{j=1}^{n+m}\gamma_{j}^{\lambda_{j}}}{\prod_{j=1}^{n+m}%
\gamma_{j}^{j}}\ \ \ \ \ \ \text{where }\lambda\in S\ \text{such that }%
\lambda_{\mu_{i}}:=\mu_{i}^{\prime}\ \text{and }\lambda_{\nu_{i}}:=\nu
_{i}^{\prime}\\
&  \geq0
\end{align*}
Hence $M_{\pi}\geq M_{\pi^{\prime}}.$

\item If $P=M_{\pi}$ then $\gamma_{1}\geq\gamma_{2}\geq\cdots\geq\gamma
_{n+m}.$

We will prove the contrapositive. Assume that it is false that $\gamma_{1}%
\geq\gamma_{2}\geq\cdots\geq\gamma_{n+m}.$ We need to show that $P>M_{\pi}$.
Let $k$ be such that $\gamma_{k}<\gamma_{k+1}.$ Let $i$ and $j$ such that
$k=\pi_{i}$ and $k+1=\pi_{j}$. We consider the following four potential cases.

\begin{enumerate}
\item $i\leq n$ and $j\leq n$

Note that $\pi$ has the following form
\[
\pi=\left(  \nu,\mu\right)  =(\,\,\underbrace{\ldots,k,k+1,\ldots}%
_{n}\ \underbrace{\ldots\vphantom{,}\ldots\newline\newline}_{m}\,\,)
\]
where $k$ appears at the $i$-th position and $k+1$ appears at the $j$-th
position (in fact, $j$ must be $i+1$, since $\pi\in S^{\ast}$) in the $\nu$
block. Thus $\gamma_{k}=\beta_{i}$ and $\gamma_{k+1}=\beta_{j}.$ Since
$\gamma_{k}<\gamma_{k+1}$, we should have $\beta_{i}<\beta_{j}$. However this
is not possible due to the global assumption $\beta_{1}>\beta_{2}>
\cdots>\beta_{n}.$ Thus this case cannot occur.

\item $i>n$ and $j>n$

This case cannot occur, due to the essentially same reason as above.

\item $i\leq n$ and $j>n$

We divide the proof into several steps.

\begin{enumerate}
\item Note that $\pi$ has the following form%
\[
\pi=\left(  \nu,\mu\right)  =(\,\,\underbrace{\ldots,k,\ldots}_{n}%
\ \underbrace{\ldots,k+1,\ldots}_{m}\,\,)
\]
where $k$ appears on the $i\,$-th position in the $\nu$ block and $k+1$
appears on the $\left(  j-n\right)  $-th position in the $\mu$ block.

\item Note $\gamma_{k}=\beta_{i}$ and $\gamma_{k+1}=\alpha_{j-n}.$ Since
$\gamma_{k}<\gamma_{k+1,}$ we have $\beta_{i}<\alpha_{j-n}.$

\item Let $\pi^{\prime}$ be obtained from $\pi$ by swapping $\pi_{i}$ and
$\pi_{j}.$ Then $\pi^{\prime}$ has the following form%
\[
\pi^{\prime}=\left(  \nu^{\prime},\mu^{\prime}\right)
=(\,\,\underbrace{\ldots,k+1,\ldots}_{n}\ \underbrace{\ldots,k,\ldots}%
_{m}\,\,)
\]
where $k+1$ appears on the $i\,$-th position in the $\nu^{\prime}$ block and
$k$ appears on the $\left(  j-n\right)  $-th position in the $\mu^{\prime}$ block.

\item We will show that $\pi^{\prime}\in S^{\ast}.$ Since $\pi\in S^{\ast},$
we have that $\nu$ and $\mu$ are strictly increasing. By inspecting the form
of $\pi$ shown above, we see that, in the $\nu$ block, everything to the left
of $k$ is less then $k$ and everything to the right of $k$ is greater than
$k+1$ and that, in the $\mu$ block, everything to the left of $k+1$ is less
than $k$ and everything to the right of $k+1$ is greater than $k+1.$ By
inspecting the form of $\pi^{\prime}$ shown above, we see that $\nu^{\prime}$
and $\mu^{\prime}$ are strictly increasing. Thus $\pi^{\prime}\in S^{\ast}.$

\item From Lemma \ref{lem:pq}, we have
\[
\frac{M_{\pi^{\prime}}}{M_{\pi}}=\frac{\cdots\alpha_{j-n}^{n-\left(  k-\left(
j-n\right)  \right)  }\ \ \cdots\ \ \ \ \cdots\beta_{i}^{m-\left(
k+1-i\right)  }\cdots}{\cdots\alpha_{j-n}^{n-\left(  k+1-\left(  j-n\right)
\right)  }\cdots\ \ \ \cdots\beta_{i}^{m-\left(  k-i\right)  }\ \ \ \ \cdots
}=\frac{\alpha_{j-n}}{\beta_{i}}>0
\]
Thus $M_{\pi^{\prime}}>M_{\pi}$. Thus $P>M_{\pi}.$
\end{enumerate}

\item $i>n$ and $j\leq n$

We can show that $P>M_{\pi}$, using the essentially same argument as above.
\end{enumerate}
\end{enumerate}
\end{proof}

\begin{notation}
Let $\Delta=\left\{  \pi\in S:P=M_{\pi}\right\}  ,$ and $\Theta=\left\{
\pi\in S^{*}:P^{*}=M_{\pi}\right\}  .$ Let $\delta=\#\Delta,$ i.e. the number
of \textquotedblleft maximum permutations\textquotedblright. Similarly, let
$\theta=\#\Theta$.
\end{notation}

\begin{lemma}
\label{lem:NP-NPstar} $\# E_{\mathfrak{R}}\left(  a,b\right)  =\delta=\theta$.
\end{lemma}

\begin{proof}
By Lemma~\ref{lem:RPPstar} we have that $P=P^{*}$. Thus, $\Theta
\subset\Delta$. Moreover, by Lemma \ref{lem:no_zigzag}, we get that
$\Delta\subset\Theta$. Therefore, $\Delta=\Theta$ and in turn $\delta=\theta$.
On the other hand, we have that $E_{\mathfrak{R}}(a,b)\subset\Delta$, and by
Lemma \ref{lem:unique} we have $\Theta\subset E_{\mathfrak{R}}(a,b)$. Thus
$\Theta\subset E_{\mathfrak{R}}(a,b)\subset\Delta$ in turn $\theta
\leq\#E_{\mathfrak{R}}(a,b) \leq\delta$. Now, the lemma follows directly.
\end{proof}

\begin{lemma}
\label{lem:ncr<->nmp}The following two are equivalent:

\begin{enumerate}
\item $A$ and $B$ have exactly $k$ common roots.

\item $\delta=2^{k}.$
\end{enumerate}
\end{lemma}

\begin{proof}
Let us prove that $1\Longrightarrow2.$ Let $A$ and $B$ be of degrees $m$ and
$n$ with exactly $k$ common roots, say $\alpha_{i_{1}}=\beta_{j_{1}}%
,\ldots,\alpha_{i_{k}}=\beta_{j_{k}}$ where the roots $\alpha$'s and $\beta$'s
are ordered as follows.%
\begin{equation}
\cdots>%
\begin{array}
[c]{c}%
\alpha_{i_{1}}\\
\parallel\\
\beta_{j_{1}}%
\end{array}
>\cdots>%
\begin{array}
[c]{c}%
\alpha_{i_{2}}\\
\parallel\\
\beta_{j_{2}}%
\end{array}
>\cdots\cdots>%
\begin{array}
[c]{c}%
\alpha_{i_{k}}\\
\parallel\\
\beta_{j_{k}}%
\end{array}
>\cdots\label{root_order}%
\end{equation}
where $\cdots$ represent strict orderings among the other (non-common) roots.
Note%
\begin{align*}
\delta &  =\#\left\{  \pi\in S:P=M_{\pi}\right\} \\
&  =\#\left\{  \pi\in S:\text{the elements in }\pi\left(  \beta,\alpha\right)
\ \text{are non-increasing}\right\}  \ \ \text{from Lemma \ref{lem:sort}}\\
&  =\#\left\{  \pi\in S:\pi\left(  \beta,\alpha\right)  =\left(  \cdots,%
\begin{array}
[c]{c}%
\alpha_{i_{1}},\beta_{j_{1}}\\
\beta_{j_{1}},\alpha_{i_{1}}%
\end{array}
,\cdots,%
\begin{array}
[c]{c}%
\alpha_{i_{2}},\beta_{j_{2}}\\
\beta_{j_{2}},\alpha_{i_{2}}%
\end{array}
,\cdots\cdots,\,%
\begin{array}
[c]{c}%
\alpha_{i_{k}},\beta_{j_{k}}\\
\beta_{j_{k}},\alpha_{i_{k}}%
\end{array}
,\cdots\right)  \right\}  \ \ \text{from\ (\ref{root_order})}\\
&  \text{(where }%
\begin{array}
[c]{c}%
\square\\
\triangle
\end{array}
\ \text{means \textquotedblleft either }\square\ \text{or }\triangle
\text{\textquotedblright)}\\
&  =2^{k}%
\end{align*}
Let us show that $2\Longrightarrow1.$ Assume that $\delta=2^{k}$. Let
$\lambda$ be the number of common roots of $A$ and $B$. Then, since
$1\Longrightarrow2$, we have $\delta=2^{\lambda}.$ Thus $2^{\lambda}=2^{k}$,
and hence $\lambda=k.$ Thus $A$ and $B$ have exactly $k$ common roots.
\end{proof}

\begin{proof}
[Proof of Main Result (Theorem~\ref{thm:main})] It is a direct consequence of
Lemmas \ref{lem:NP-NPstar} and \ref{lem:ncr<->nmp}.
\end{proof}

\section{Conclusion and Discussion}

\label{sec:conclude} The goal of this paper was to adapt the following well
known property of univariate resultant over $\mathbb{C}$ to the tropical
semifield: the number of common roots of the two polynomials is the same as
the order of the resultant at the tuple of the coefficients. We have shown
that the same property holds if we adapt the notion of order, and we restrict
the roots of the polynomials to be tropical non-zeros and simple.

In the following, we will briefly and informally discuss a few questions
naturally raised by the results given in this paper.

\begin{enumerate}
\item \emph{Conditions on the root}. Note that we treated only the case when
all the roots are tropical non-zeros and simple. We discuss what happens in the
other cases.

\begin{enumerate}
\item \emph{tropical zero} root: Suppose that the two polynomials $A$ and $B$
have  the tropical zero $(-\infty)$ as a root. The main result does not hold. The reason is as
follows: each term of the polynomial $\mathfrak{R}$ always contains, at least,
one of the indeterminates~$\mathbf{a}_{m}$ and~$\mathbf{b}_{n}$. On the other
hand, the fact that $-\infty$ is a common root of $A,B$ implies that
$a_{m}=b_{n}=-\infty.$ Thus, for every term $t$ in $\mathfrak{R},$ we have
$\mathfrak{R}(a,b)=t_{i}(a,b)=-\infty$. Therefore the order can be different
from (in fact \emph{not} related to) the number of common roots.
Hence, in order to cover tropical zero roots, one will need to come up with a different notion of order. We leave it as an open challenge.
\item \emph{multiple} root: Suppose that a polynomial has a multiple (not
simple) root. The main result does not hold. The reason is as follows: a
function with a multiple root admits infinitely many polynomial
representations. Furthermore the order of the resultant depends on which
representation is chosen. Therefore the order can be different from the number
of common roots. For example, consider the following three polynomials:
\begin{align*}
A  & =0\mathbf{x}+3  \\
B_{1}  & =0\mathbf{x}^{2}+2\mathbf{x}+6\\
B_{2}  & =0\mathbf{x}^{2}+3\mathbf{x}+6
\end{align*}
representing the following functions:
\begin{center}
\includegraphics[width=3cm]{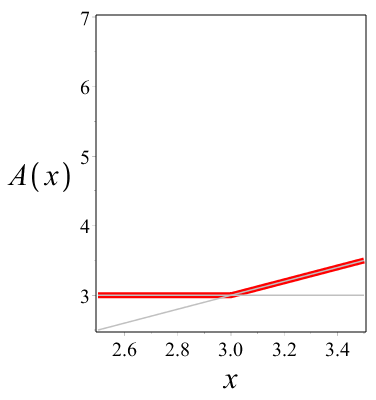}
\includegraphics[width=3cm]{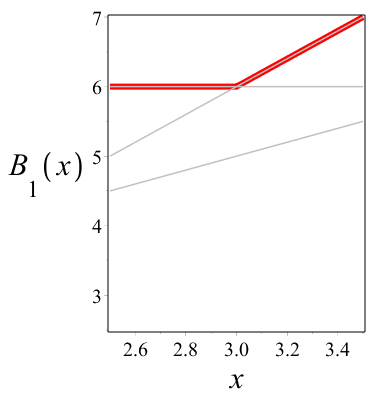}
\includegraphics[width=3cm]{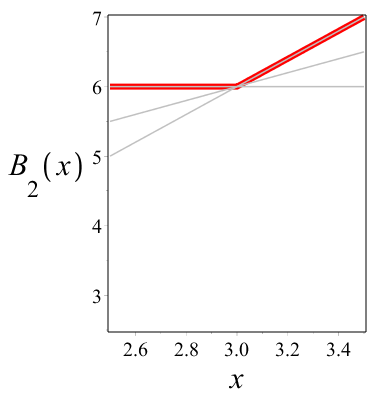}
\end{center}

\begin{itemize}
\item
It is obvious that $A$ has a simple root, namely $3$.
It is also obvious that $B_{1}$ and $B_{2}$ represent the same
function, with one double root, namely $3$.
\item
Thus the number of common roots of $A$ and $B_1$ is $1$. Likewise the
number of common roots of $A$ and $B_2$ is also $1$.
\item Direct computation show that
\[
\mathfrak{R}(\mathbf{a},\mathbf{b})  =
\mathbf{a}_{1}^{2}\mathbf{b}_0 +
\mathbf{a}_0\mathbf{a}_{1}\mathbf{b}_{1} +
\mathbf{a}_0^2\mathbf{b}_{2}
\]
Thus
\begin{eqnarray*}
\mathfrak{R}(((0,3),(0,2,6))) &=& \max\{6,5,6\} \\
\mathfrak{R}(((0,3),(0,3,6))) &=& \max\{6,6,6\}
\end{eqnarray*}
\item Hence
\[O_{\mathfrak{R}}((0,3),(0,2,6))=\log_{2}2=1\]
 but
\[O_{\mathfrak{R}}((0,3),(0,3,6))=\log_{2}3\neq 1\]
\end{itemize}
Hence, in order to cover multiple roots, one will need to come up with a different notion of order. One natural and potential   approach  might be to  look at  the variety (polyhedral fan complex) of the resultant and investigate the co-dimension of  the cone where the coefficients vector of two polynomials $A$ and $B$ belongs to. We
leave it as an open challenge.

\end{enumerate}

\item \emph{Rank deficiency}. Over $\mathbb{C}$, it is well known that the
number of the common roots is the same as the rank deficiency of the Sylvester
matrix. Thus, one wonders whether the relation can be adapted to the tropical
semifield. We divide the discussion into two cases.

\begin{enumerate}
\item \emph{All roots are tropical non-zeros and simple}: Through numerous
experiments on computer, we conjecture that the relation holds in this case if
the rank is taken as the tropical rank~\cite{DevSanSt05,IR2009}. It is
relatively easy to prove that the number of the common roots is bounded
\emph{below} by the rank deficiency (exploiting some of the proof techniques
developed in this paper). However, it seems challenging to prove/disprove that
the number of the common roots is bounded \emph{above} by the rank deficiency.
We leave it as open challenge.

\item \emph{The other cases}: The relation does not hold in general, as
illustrated by the   example
\[
A  =\mathbf{x}^{3}+2\mathbf{x}^{2}+2\mathbf{x}+6, \;\;\;\;\;\;
B  =\mathbf{x}^{2}-2\mathbf{x}+4
\]
It is easy to verify that the roots of $A$ and $B$ are respectively
$(2,2,2)$ and  $(2,2)$.
Thus, the number of common roots is $2$.
However, a direct computation shows that the tropical
rank deficiency of the Sylvester matrix is $1$. In \cite{DevSanSt05}, two
other different notions of ranks are considered, namely Kapranov and
Barvinok. Hence one wonders whether any of those might make the relation hold.
However, according to Theorem 1.4 in the paper,  the tropical rank is not greater
than the other two, and hence the tropical rank deficiency based on the other
two can never be greater than $1$. Thus, the relation does not hold for
Kapranov and Barvinok ranks either.

Hence, in order to cover  the  tropical zero roots or multiple roots, one will need to come up with a different notion of rank. We leave it as an open challenge.

\end{enumerate}

\end{enumerate}

\vspace*{0.5cm}

\noindent{\large \textbf{Acknowledgement.}} H. Hong (hong@ncsu.edu) is
partially supported by US NSF 1319632. J.R. Sendra (rafael.sendra@uah.es) is
partially supported by the Spanish Ministerio de Econom\'{\i}a y
Competitividad under the Project MTM2014-54141-P.

\bibliographystyle{plain}
\bibliography{paper}

\section{Appendix: number of common roots and resultants over $\mathbb{C}$ (by
Laurent Bus\'e)}

Recall the following property mentioned in the introduction: the order of the
point at the resultant is equal to the number of common complex roots of the
two polynomials over $\mathbb{C}$. This property is definitely part of the
folklore but we were not able to find it in the existing literature. In the
following, we communicate a simple proof kindly provided by Laurent Bus\'e
(laurent.buse@inria.fr).  The proof is presented in a bit more general context
of unique factorization domain. Furthermore, the number of common roots is
seen as the degree of gcd. \bigskip

Let $k$ be a unique factorization domain. Given two positive integers $m,n$,
consider the homogeneous polynomials
\begin{align*}
f(x,y)  & = a_{0}x^{m}+a_{1}x^{m-1}y+\cdots+a_{m}y^{m}\\
g(x,y)  & = b_{0}x^{n}+b_{1}x^{n-1}y+\cdots+b_{n}y^{n}%
\end{align*}
in the variables $x,y$ with coefficients in the ring $\mathbb{A}%
:=k[a_{0},\ldots,a_{m},b_{0},\ldots,b_{n}].$
The Sylvester resultant $R:=\operatorname{Res}(f,g)$ of $f(x,y)$ and $g(x,y)$
is a polynomial in $\mathbb{A}$.
Given a point
\[
s:=(p_{0},\ldots,p_{m},q_{0},\ldots,q_{n}) \in k^{m+n+2},
\]
the question is to determine the order of the resultant at $s$.

\begin{proposition}
The order of the resultant polynomial $R$ at the point $s$ is equal to the
degree of the gcd of the polynomials
\[
p(x,y)=\sum_{i=0}^{m} p_{i}x^{m-i}y^{i}, \ q(x,y)=\sum_{i=0}^{n} q_{i}%
x^{n-i}y^{i},
\]
unless $s=(0,\ldots,0)$, i.e.~$(p,q)=(0,0)$, in which case the order is equal
to $m+n$.
\end{proposition}

\begin{proof}
The order of the resultant $R$ at $s$ is nothing but the $t$-valuation of the
polynomial
\[
R(p_{0}+ta_{0},\ldots,p_{m}+ta_{m},q_{0}+tb_{0}+\ldots,q_{n}+tb_{n}%
)=\operatorname{Res}(p(x,y)+tf(x,y),q(x,y)+tg(x,y))\in\mathbb{A}[t].
\]
Denote by $h(x,y)$ the gcd of $p(x,y)$ and $q(x,y)$ and by $\delta$ its
degree; there exist two polynomials $\tilde{p}$ and $\tilde{q}$ such that
$p=\tilde{p}h$ and $q=\tilde{q}h$.

If $p=q=0$ then the claimed property is clear. If $q=0$ and $p\neq0$, then
\[
\operatorname{Res}(p+tf,q+tg)=\operatorname{Res}(p+tf,tg)=t^{m}%
\operatorname{Res}(p+tf,g).
\]
Since $\operatorname{Res}(p+tf,g)_{|t=0}=\operatorname{Res}(p,g)\neq0$, for
$p\neq0$ and $g$ is the generic homogeneous polynomial of degree $n$, the
claimed property is proved.

Now, assume that $q\neq0$. By applying some classical properties of the
resultant, we have that
\begin{multline*}
\operatorname{Res}(\tilde{q},tg)\operatorname{Res}(\tilde{p}h+tf,\tilde
{q}h+tg)= \operatorname{Res}(\tilde{q},\tilde{q}h+tg)\operatorname{Res}%
(\tilde{p}h+tf,\tilde{q}h+tg)\\
= \operatorname{Res}(\tilde{q}\tilde{p}h+t\tilde{q}f,\tilde{q}h+tg)
=\operatorname{Res}(t(\tilde{q}f-\tilde{p}g),\tilde{q}h+tg).
\end{multline*}
It follows that
\[
t^{n-\delta}\operatorname{Res}(\tilde{q},g)\operatorname{Res}(p+tf,q+tg)=
t^{n}\operatorname{Res}(\tilde{q}f-\tilde{p}g,q+tg).
\]
From here, the claimed property follows since $\operatorname{Res}(\tilde
{q},g)\neq0$ and $\operatorname{Res}(\tilde{q}f-\tilde{p}g,q)\neq0$, for
$\tilde{p}$ and $\tilde{q}$ are coprime polynomials.
\end{proof}

\end{document}